\documentclass[a4paper]{amsart}
\newtheorem{thm}{Theorem}[section]

\newtheorem{lem}[thm]{Lemma}
\newtheorem{prop}[thm]{Proposition}

\newtheorem{problem}[thm]{Problem}

\theoremstyle{definition}

\theoremstyle{remark}

\numberwithin{equation}{section}

\newcommand{\BBR}{\mathbb R}
\DeclareMathSymbol{\emptyset}{\mathord}{AMSb}{"3F}
\DeclareMathSymbol{\preccurlyeq}{\mathrel}{AMSa}{"34}

\newcommand{\cc}{\mathop{\mathrm{cc}}}
\newcommand{\conv}{{\mathbf{Conv}}}
\newcommand{\comp}{{\mathbf{Comp}}}
\newcommand{\cone}{\mathop{\mathrm{cone}}}
\newcommand{\pr}{\mathop{\mathrm{pr}}}
\def\w{\mathrm{w}}

\def\CS{\mathcal{S}}

\usepackage[all]{xy}
\begin{document}

\title[]{Hyperspace of convex compacta of nonmetrizable compact convex subspaces of locally convex spaces}
\author{Lidia Bazylevych}
\address{National University ``Lviv Polytechnica", 12 Bandery Str.,
79013 Lviv, Ukraine} \email{izar@litech.lviv.ua}

\author{Du\v san Repov\v s}
\address{Institute of Mathematics, Physics and Mechanics, and
Faculty of Education,
University of Ljubljana,
P.O.Box 2964,
Ljubljana, Slovenia 1001}
\email{dusan.repovs@guest.arnes.si}

\author{Michael Zarichnyi}
\address{Department of Mechanics and Mathematics,
Lviv National University, Universytetska Str. 1, 79000 Lviv,
Ukraine}
\address{Institute of Mathematics,
University of Rzesz\'ow, Rejtana 16 A, 35-310 Rze\-sz\'ow, Poland}
\email{topology@franko.lviv.ua}
\thanks{The research was supported by the Slovenian-Ukrainian grant
BI-UA/04-06-007. The authors are indebted to the referee for valuable remarks.}
\subjclass[2000]{54B20, 54C55, 46A55} \keywords{Compact convex
set, hyperspace, Tychonov cube, soft map}
\date{\today}
\dedicatory{}
\commby{}


\begin{abstract} Our main result states that the hyperspace of
convex compact subsets of a compact convex subset $X$ in a locally
convex space is an absolute retract if and only if $X$ is an
absolute retract of weight $\le\omega_1$. It is also proved that
the hyperspace of convex compact subsets of the Tychonov cube
$I^{\omega_1}$ is homeomorphic to $I^{\omega_1}$. An analogous
result is also proved for the cone over $I^{\omega_1}$. Our proofs
are based on analysis of maps of hyperspaces of compact convex
subsets, in particular, selection theorems for such maps are
proved.
\end{abstract}

\maketitle

\section{Introduction}

For any uncountable cardinal number $\tau$, the Tychonov and the Cantor cubes (denoted by $I^\tau$ and $D^\tau$,
respectively), belong to
the class of main geometric objects in the topology of
non-metrizable compact Hausdorff spaces. The spaces $I^\tau$ (we denote
$I$  the segment $[0,1]$) and $D^\tau$ were first
characterized by Shchepin \cite{S}. In particular, the
Tychonov cubes are characterized as the homogeneous-by-character
nonmetrizable compact Hausdorff absolute retracts \cite{S1}. This
characterization was later applied to the study of topology of the
{\em functor-powers}, i.e. spaces of the form $F(K^\tau)$, where
$K$ is a compact metrizable space and $F$ is a covariant functor
in the category of compact Hausdorff spaces. In particular, it was
proved that, for uncountable $\tau$,  the space $P(I^\tau)$, where
$P$ denotes the probability measure functor, is homeomorphic to
$I^\tau$ if and only if $\tau=\omega_1$. For the hyperspace
functor $\exp$ it is known that $\exp(D^\tau)$ is homeomorphic to
$D^\tau$ if and only if $\tau=\omega_1$ and $\exp(I^\tau)$ is not
an absolute retract whenever $\tau>\omega$.

 In this paper we consider the hyperspaces  $\cc(X)$ of  nonempty compact convex
subsets  in $X$ for compact convex subsets $X$ in locally convex
spaces.  For metrizable $X$, this object was investigated by
different authors (see, e. g. \cite{NQS}, \cite{M}). In
particular, it was proved in \cite{NQS} that the hyperspace of
convex compact subsets of the Hilbert cube $Q=I^\omega $ is
homeomorphic to $I^\omega $.

The aim of this paper is to consider the nonmetrizable compact
convex subsets in locally convex spaces. One of our main results
is Theorem \ref{t:1}, which characterizes the compact convex
spaces $X$ with $\cc(X)$ being an absolute retract. We also show
that the space $\cc(X)$ is homeomorphic to $I^{\omega_1}$
(resp.  the cone over $I^{\omega_1}$) if and only if $X$ is
homeomorphic to $I^{\omega_1}$ (resp. the cone over $I^{\omega_1}$).

These results are in the spirit of the corresponding results
concerning the functor-powers of compact metric spaces (see
\cite{S}). The proofs are based on the spectral analysis of
nonmetrizable compact Hausdorff spaces, in particular on the
Schepin Spectral Theorem \cite{S} as well as on analysis of the
selection type properties of the maps of the hyperspaces of
compact convex subsets.

The construction $\cc$ determines a functor acting on the category
$\conv$ of compact convex subsets of locally convex spaces. The
results of this paper  demonstrate that the functor $\cc $ is
closer to the functor $P$ of probability measures than to the
hyperspace functor $\exp$.

\section{Preliminaries}
All topological spaces are assumed to be Tychonov, all maps are
continuous. By $\bar A$ we denote the closure of a subset $A$ of a
topological space. Let $X$ be any space.

The {\it hyperspace\/} $\exp X$ of $X$ is the space of all
nonempty compact subsets in $X$ endowed with the  {\it
Vietoris\/} topology. A base of this topology is formed by the
sets of the form
$$
\langle U_1,\dots,U_n\rangle=\bigl\{A\in\exp X\mid A\subset
U_1\cup\dots\cup U_n\text{ and }A\cap U_i\neq\emptyset\text{ for every
}i\bigr\},
$$
where $U_1,\dots,U_n$ run through the topology of $X$,
$n\in\mathbb N$. For a metric space $(X,\rho)$ the Vietoris
topology on $\exp(X)$ is induced by the Hausdorff metric
$\rho_{\mathrm{H}}$:
$$\rho_{\mathrm{H}}(A,B)=\inf\{\varepsilon>0
\mid A\subset O_{\varepsilon}(B),\ B\subset O_{\varepsilon}(A)\}.$$

The hyperspace construction determines a functor in the category
$\comp$ of compact Hausdorff spaces and continuous maps. Given a
map $f\colon X\to Y$ in $\comp$, we define $\exp(f)\colon
\exp(X)\to\exp(Y)$ by $\exp(f)(A)=f(A)$, $A\in \exp(X)$.

Let $\conv$ denote the category of compact convex subsets in
locally convex spaces and affine continuous maps. If $X$ is an
object of $\conv$ we define $$\cc(X)=\{A\in\exp(X)\mid A\mbox{ is
convex}\}\subset\exp(X).$$ If $f\colon X\to Y$ is a map in
$\conv$, then the map $\cc(f)\colon \cc(X)\to \cc(Y)$ is defined
as the restriction of $\exp(f)$ on $\cc(X)$.

In the sequel, for a nonempty compact subset $X$ in a locally
convex space $Y$,  we denote
the~closed convex hull map
by $h\colon \exp X\to\cc(Y)$.
Let $X$ be a subset of a metrizable locally convex space $M$. In
the sequel, we identify any point $x\in X$ with the singleton
$\{x\}\in\cc(X)$.

Recall that the {\it Minkowski operation\/} in $\cc(X)$ is defined
as follows:
$$\lambda_1A_1+\lambda_2A_2=\{\lambda_1x_1+\lambda_2x_2\mid x_1\in
A_1,\ x_2\in A_2\},$$ $\lambda_1,\lambda_2\in\BBR$,
$A_1,A_2\in\cc(X)$.

\begin{lem} Let $X$ be a  compact convex subset in a
locally convex space. There exists an~embedding $\alpha$ of the
space $\cc(X)$ into a~locally convex space $L$ satisfying the
condition
\begin{equation}\label{f:emb}
 \alpha(\lambda_1A_1+\lambda_2A_2)=\lambda_1\alpha(A_1)+\lambda_2\alpha(A_2)
\end{equation}
for every $\lambda_1,\lambda_2\in\BBR$, $A_1,A_2\in\cc(X)$.
\end{lem}
\begin{proof} Let $X$ be a compact convex subset in a metrizable
locally convex space $M$. Following \cite{P}, consider the
equivalence relation $\sim$ on $\cc(M)\times\cc(M)$ defined by the
condition: $(A,B)\sim(C,D)$ if and only if $A+D=B+C$. Denote by $L$
the set of equivalence classes of $\sim$ (in the sequel, we denote
by $[A,B]$ the equivalence class that contains $(A,B)$). It is
well-known that $L$ is a linear space with respect to the naturally
defined operations. Let $U$ be a convex neighborhood of zero in
$M$ and define
$$U^\ast=\{[A,B]\in L\mid A\subset B+U,\ B\subset A+U\}.$$
The sets
$U^\ast$ form a base at
zero in $L$. The map $\alpha\colon\cc(X)\to
L$ defined by the formula $\alpha(A)=[A,\{0\}]$ is the
required
embedding.
\end{proof}

\section{Functor $\cc$ and soft maps}

A map $f\colon X\to Y$ is {\it soft} (see \cite{S}) if for every
commutative diagram $$\xymatrix{A \ar[r]^{\psi}\ar[d]_i &
X\ar[d]^{f} \\ Z\ar[r]_{\varphi} & Y,}$$ where $i\colon A\to Z$ is
a closed embedding into a paracompact space $Z$, there exists a
map $\Phi\colon Z\to X$ such that $\Phi | A=\psi$ and
$f\Phi=\varphi$.

In other words, a map is soft if it satisfies the parameterized
selection extension property.

The following proposition is close to the Michael selection
theorem for convex-valued maps \cite{Mi}.
\begin{prop}\label{p:soft} Let $f\colon X\to Y$ be an~affine open map of compact
convex metrizable subsets of locally convex spaces. Then the~map
$\cc(f)\colon \cc(X)\to\cc(Y)$ is soft.
\end{prop}
\begin{proof}
We first prove that the~map $\cc(f)$ is open. It is well-known that
the~map $\exp(f)$ is open. Since the diagram
$$\xymatrix{{(\exp(f))^{-1}(\cc(Y))}\ar[rr]^h
\ar[dr]_{{\exp(f)|(\exp(f))^{-1}(\cc(Y))}\ \ \ \ \ \ }&&{\cc(X)}\ar[dl]^{\cc(f)}\\
&{\cc(Y)}&}$$ is commutative and the closed convex hull map $h$ is
a retraction of $(\exp(f))^{-1}(\cc(Y))$ onto $\cc(X)$, we see
that the~map $\cc(f)$ is also open.

There exists an~embedding $\alpha\colon\cc(X)\to L$ satisfying
condition (\ref{f:emb}). Choose a~countable family of~functionals
$\{\varphi_1,\varphi_2,\dots\}\subset L^*$ such that this family
separates the points and
$\varphi_i(\alpha(\cc(X)))\subset[-1/i,1/i]$. Then the~map
$\varphi=(\varphi_1,\varphi_2,\dots)$, defined on
$\alpha(\cc(X))$, embeds $\alpha(\cc(X))$ into the~Hilbert space
$\ell^2$. Denote by $$\xi\colon
\varphi(\alpha(\cc(X)))\times\cc(\varphi(\alpha(\cc(X))))\to
\varphi(\alpha(\cc(X)))$$ the~nearest point map: $y=\xi(x,A)$ if
and only if $\|z-x\|>\|y-x\|$, for every $z\in A\setminus\{y\}$
(here $\|\cdot\|$ denotes the~norm in $\ell^2$).

Suppose a~commutative diagram
$$\xymatrix{A \ar[r]^{p}\ar@{^{(}->}[d] & {\cc(X)}\ar[d]^{\cc(f)} \\
Z\ar[r]_{q} & {\cc(Y)}}$$ is given, where $A$ is a closed subset of
a~paracompact space $Z$.

Since $\cc(X)$ is an~absolute retract, there exists a~map $r\colon
Z\to \cc(X)$ such that $r|A=p$. Note that for every $B\in\cc(Y)$,
the~set $\varphi(\alpha(\cc(f)^{-1}(B)))$ is a~convex closed subset
of $\varphi(\alpha(\cc(X)))$, i.e. an~element of the~space
$\cc(\varphi(\alpha(\cc(X))))$. Since the~map $\cc(f)$ is open,
the~map
$$\delta\colon\cc(Y)\to\cc(\varphi(\alpha(\cc(X)))),\
\delta(B)=\varphi(\alpha(\cc(f)^{-1}(B))),$$ is continuous.

Define the~map $R\colon Z\to \cc(X)$ by the~formula
$$R(z)=\alpha^{-1}(\varphi^{-1}(\xi(\varphi(\alpha (r(z))),\delta(q(z))))),\ z\in Z.$$

It is easy to see that $R$ is continuous, $R|A=p$, and $\cc(f)R=q$.
\end{proof}

A point $p$ of a set $X$ in a locally convex space $E$ is called
an {\em exposed point} of $X$ if there exists a continuous linear
functional $f$ on $E$ such that $f(x)> f(p)$, for each $x \in
X\setminus\{p\}$.

\begin{lem}\label{l:3}
Let $f\colon X\to Y$ be an open affine continuous map of compact
convex subsets in locally convex spaces such that $|f^{-1}(y)|>1$
for every $y\in Y$. Then $|\cc(f)^{-1}(B)|>1$, for every $B\in
\cc(Y)$.
\end{lem}
\begin{proof} As in the proof of Proposition \ref{p:soft}, one
may assume that $X$ is affinely embedded in the Hilbert space
$\ell^2$.  Let $B\in \cc(Y)$ and $A\in\cc(f)^{-1}(B)$. If $A\neq
f^{-1}(B)$, then we
define $A'$ as the closure of the convex hull of
$A\cup\{x\}$, where $x\in f^{-1}(B)\setminus A$. Then $A'\neq A$
and $A'\in\cc(f)^{-1}(B)$.

If $A= f^{-1}(B)$, then it is well-known (see, e.g. \cite{BL})
there exists an exposed point, $x$ of $A$. Since $f$ is open,
there exists a neighborhood $U$ of $x$ such that $f(A\setminus
U)=B$. In this case we
define $A'$ as the closure of the convex hull
of $A\setminus U$. Note that $A'\in \cc(f)^{-1}(B)$. That $A\neq
A'$ easily follows from the fact that $x$ is an exposed point.
\end{proof}

\begin{lem}\label{l:op}
Suppose that $f\colon X\to Y$ is a continuous affine map of compact
convex subsets of locally convex spaces. If the map $\cc(f)$ is
open then so is the map $f$.
\end{lem}
\begin{proof} Suppose the contrary,  that $f$ is not open. Then
there exists $x\in X$ and a  net
$(y_\alpha)_{\alpha\in A}$ in $Y$ converging to $y=f(x)$,
such that there is no  net $(x_\alpha)_{\alpha\in A}$ in $X$ converging
to $x$
with $x_\alpha\in
f^{-1}(y_\alpha)$, for every $\alpha\in A$.

Assuming that the map $\cc(f)$ is open, we obtain that there exists
a net $(C_\alpha)_{\alpha\in A}$ in $\cc(X)$ converging to $\{x\}$
and such that $\cc(f)(C_\alpha)=\{y_\alpha\}$, for every $\alpha\in
A$. Then, obviously, the net $(c_\alpha)_{\alpha\in A}$ converges
to $x$, for every choice $c_\alpha\in C_\alpha$, $\alpha\in A$. This
gives a contradiction.
\end{proof}

A commutative diagram
\begin{equation}\label{f:1}
\mathcal{D}={\xymatrix{X\ar[r]^f\ar[d]_g&
Y\ar[d]^u\\ Z\ar[r]_v& T}}
\end{equation} is called {\it soft} if its {\it
characteristic map} $$\chi_{\mathcal{D}}=(f,g)\colon X\to
Y\times_TZ=\{(y,z)\in Y\times Z\mid u(y)=v(z)
\}$$ is soft.
\begin{lem}\label{l:soft}
Suppose that a commutative diagram $\mathcal{D} $ (see formula
(\ref{f:1})) in the category $\conv$ consists of metrizable spaces.
If the diagram $\cc(\mathcal{D}) $ is soft, then so is the diagram
$\mathcal{D} $.
\end{lem}
\begin{proof} First we show that  the diagram $\mathcal{D}$ is
open if such
is $\cc(\mathcal{D}) $. Let $(y_i,z_i)_{i=1}^\infty$ be
a sequence in $Y\times_TZ$ converging to a point $(y,z)$ and let
$x\in X$ be such that $\chi_\mathcal{D}(x)=(y,z)$. Since
$\cc(\mathcal{D})$ is soft (and therefore open), there exists a
sequence $(A_i)_{i=1}^\infty$ in $\cc(X)$ such that
$(f(A_i),g(A_i))=(\{y_i\},\{z_i\})$, for every $i$, and
$(A_i)_{i=1}^\infty$ converges to $\{x\}$ in $\cc(X)$. Choose
arbitrary $x_i\in A_i$, then $(f(x_i),g(x_i))=(y_i,z_i)$, for
every $i$, and $(x_i)_{i=1}^\infty$ converges to $x$ in $X$. This
shows that the map $\chi_{\mathcal{D}}$ is open.

Now the map $\chi_{\mathcal{D}}$, being an open affine map of
convex compact metrizable subspaces of locally convex spaces, is
soft. This follows from the Michael Selection Theorem \cite{Mi}
(see e.g. \cite{S}).
\end{proof}
\section{Hyperspaces $\cc(X)$ homeomorphic to Tychonov cubes}

We are going to recall
some definitions and results related to the
Shchepin Spectral Theorem  (see \cite{S} for details).
In what follows, an {\it inverse system} $\mathcal S=
\{X_\alpha,p_{\alpha\beta};\mathcal A\}$ satisfies the~following
conditions:
\begin{enumerate}
\item[1)] $X_\alpha$ are compact Hausdorff spaces;
\item[2)] $p_{\alpha\beta}$ are surjective;
\item[3)] the partially ordered set $\mathcal A$ (by $\le$) is directed,
i.e., for every $\alpha,\beta\in\mathcal A$ there exists
$\gamma\in\mathcal A$ with $\alpha\le\gamma$, $\beta\le\gamma$.
\end{enumerate}

An inverse system $\mathcal S= \{X_\alpha,p_{\alpha\beta};\mathcal
A\}$ is called {\it open\/} if all the maps $p_{\alpha\beta}$ are
open.
An inverse system
$\CS=\{X_\alpha,p_{\alpha\beta};\mathcal A\}$ is called {\it
continuous\/} if for every $\alpha\in\mathcal A$ we have
$X_\alpha=\varprojlim\{X_{\alpha '},p_{\alpha '\beta '};\alpha
',\beta '<\alpha\}$.

By $\mathrm{w}(X)$ we denote the {\it weight\/} of a~space $X$. An
inverse system $\mathcal S=\{X_\alpha,p_{\alpha\beta};\mathcal
A\}$ is called a~$\tau$-{\it system}, $\tau$ being a~car\-di\-nal
number, if the~following holds:
\begin{enumerate}
  \item[1)] the directed set $\mathcal A$ is $\tau$-complete, i. e. every
  chain of cardinality $\le\tau$ in $\mathcal A$ has the least upper
  bound;
  \item[2)] $\mathcal S$ is continuous;
  \item[3)] $\mathrm{w}(X_\alpha)\le\tau$, for every $\alpha\in\mathcal A$.
\end{enumerate}

If $\tau=\omega$, we use the terms $\sigma$-{\it complete\/} and
$\sigma$-{\it system.\/}

For every $A$,  we denote the~family of
all countable subsets of $A$ ordered by inclusion
by $\mathcal P_\omega(A)$.

A standard way to represent a compact Hausdorff space $X$ as a
limit of a $\sigma$-system is to embed it into a Tychonov cube
$I^\tau$, for some $\tau$. For any countable $A\subset\tau$, let
$X_A=p_A(X)$, where $p_A\colon I^\tau\to I^A$ denotes the
projection. In this way we obtain an inverse system $\mathcal
S=\{X_A,p_{AB};\mathcal P_\omega(\tau)\}$, where, for $A\supset
B$,  $p_{AB}\colon X_A\to X_B$ denotes the (unique) map with the
property $p_B|X=p_{AB}(p_A|X)$. The resulting inverse system
$\mathcal S$ is a $\sigma$-system and $X=\varprojlim\mathcal S$.

If $X$ is a compact convex subset of a locally convex space, we
can affinely embed $X$ into $I^\tau$, for some $\tau$. The above
construction gives us an inverse $\sigma$-system $\mathcal S$ in
the category $\conv$ such that $X=\varprojlim\mathcal S$.

In the sequel, we will use the well-known fact that the functor
$\cc$ is continuous in the sense that it commutes with the limits
of inverse systems.

A compact Hausdorff space $X$ is {\em openly generated} if $X$ is
the limit of an inverse $\sigma$-system with open short
projections.
The {\em absolute retracts} (ARs) are considered in the class of
compact Hausdorff spaces.

\begin{thm}\label{t:1}
Let $X$ be a convex compact subset of a locally convex space. Then the
space $\cc(X)$ is an~absolute retract if and only if $X$ is openly
generated and of weight $\le\omega_1$.
\end{thm}

\begin{proof}
If $X$ is openly generated and of weight $\le\omega_1$, then $X$
is homeomorhic to $\varprojlim\mathcal S$, where $\mathcal
S=\{X_\alpha,p_{\alpha\beta};\omega_1\}$ is an inverse system
consisting of convex compact subsets of metrizable locally convex
spaces and open maps. Then $\cc(X)$ is homeomorhic to
$\varprojlim\cc(\mathcal S)$. Since the spaces $\cc(X_\alpha)$ are
ARs and the maps $\cc(p_{\alpha\beta})$ are soft (see Proposition
\ref{p:soft}), the space $\cc(X)$ is an AR.

Suppose now that $\cc(X)$ is an AR of weight $\ge\omega_2$. It
easily follows from standard results of Shchepin's theory that
there exists a compact convex space $\tilde X$ of weight
$\omega_2$ such that $\cc(\tilde X)$ is an AR (see \cite{S} and
also \cite{Ch}, where the case of locally convex spaces is
considered). We may assume that $\cc(\tilde
X)=\varprojlim\cc(\tilde{\mathcal S})$, where $\tilde{\mathcal
S}=\{\tilde X_\alpha,\tilde p_{\alpha\beta};\omega_2\}$ is an
inverse system such that for every $\alpha<\omega_2$ the space
$\cc(\tilde X_\alpha)$ is an AR and for every $\alpha,\beta$,
$\beta\le\alpha<\omega_2$, the map $\cc(\tilde p_{\alpha\beta})$
is soft. In its turn, every $\tilde X_\alpha$ can be represented
as $\varprojlim\tilde{\mathcal S}_\alpha$, where $\tilde{\mathcal
S}_\alpha=\{\tilde X_{\alpha\gamma},\tilde
q_{\gamma\delta}^\alpha;\omega_1 \}$ is an inverse system in
$\conv$ and it follows from the results of Chigogidze \cite{Ch}
that for every $\alpha,\beta$, where $\beta \le \alpha<\omega_2$,
the map $\tilde p_{\alpha\beta}$ is the limit of a morphism
$(\tilde
p_{\alpha\beta}^\gamma)_{\gamma<\omega_1}\colon\tilde{\mathcal
S}_\alpha\to\tilde{\mathcal S}_\beta$ such that the maps
$\cc(\tilde p_{\alpha\beta}^\gamma)$ are soft and for every
$\gamma\ge\delta$, $\gamma,\delta<\omega_1$, the diagram
$$\xymatrix{{\cc(\tilde
X_{\alpha\gamma})}\ar[r]^{\cc(\tilde p_{\alpha\beta}^\gamma)}
\ar[d]_{\cc(q_{\gamma\delta}^\alpha)}&
{\cc(\tilde X_{\beta\gamma})}\ar[d]^{\cc(q_{\gamma\delta}^\beta)}\\
{\cc(\tilde X_{\alpha\delta})}\ar[r]_{\cc(\tilde
p_{\alpha\beta}^\delta)} & {\cc(\tilde X_{\beta\delta})}}$$ is
soft. Since all the spaces in the above diagram are metrizable, by
Lemma \ref{l:soft}, the diagram
$$\xymatrix{{\tilde X_{\alpha\gamma}}\ar[r]^{\tilde
p_{\alpha\beta}^\gamma}
\ar[d]_{q_{\gamma\delta}^\alpha}&
{\tilde X_{\beta\gamma}}\ar[d]^{q_{\gamma\delta}^\beta}\\ {\tilde
X_{\alpha\delta}}\ar[r]_{\tilde p_{\alpha\beta}^\delta} & {\tilde
X_{\beta\delta}}}$$ is also soft. As the limits of soft morphisms,
the maps $\tilde p_{\alpha\beta}$ are soft and we conclude that the
space $\tilde X$ is an absolute retract.

Since the space $\tilde X$ is an AR, it contains a copy of the
Tychonov cube $I^{\omega_2}$. It follows from the Shchepin
Spectral Theorem that, without loss of generality, one may assume
that every $\tilde X_\alpha$ contains the space
$(I^{\omega_1})^\alpha$ and for every $\alpha,\beta$, where
$\beta\le\alpha<\omega_2$, the map $\tilde
p_{\alpha\beta}|(I^{\omega_1})^\alpha$ is the projection map of
$(I^{\omega_1})^\alpha$ onto $(I^{\omega_1})^\beta$.

Denote by $D$ the Aleksandrov supersequence of weight $\omega_1$,
i.e. the one-point compactification of a discrete space of
cardinality $\omega_1$.

\smallskip
{\bf Claim.} There exists $\alpha<\omega_2$ such that the subspace
$(I^{\omega_1})^\alpha\subset \tilde X_\alpha$ contains an affinely
independent copy of the space $D$.

\smallskip
{\it Proof of Claim.}  Represent $D$ as $\{d_\gamma\mid
\gamma\le\omega_1\}$, where $d_{\omega_1}$ denotes
 the unique non-isolated point of $D$. For $\gamma<\omega_1$,
 let $r_\gamma\colon D\to \{d_\delta\mid
\delta\le\gamma\}\cup \{d_{\omega_1}\}$ denote the retraction that
sends $\{d_\delta\mid
\gamma<\delta<\omega_1\}$ into $d_{\omega_1}$.

Define by transfinite induction maps $f_\gamma\colon D\to
(I^{\omega_1})^{\alpha_\gamma}\subset \tilde X_{\alpha_\gamma}$,
where $\gamma<\omega_1$ and $\alpha_\gamma<\omega_2$, so that
$\alpha_{\gamma}\le\alpha_{\gamma'}$ and $\tilde
p_{\alpha_{\gamma'}\alpha_\gamma}f_{\gamma'}=f_\gamma$ for every
$\gamma\le\gamma'$.

Let $f_0\colon D\to (I^{\omega_1})^{\alpha_0}\subset \tilde
X_{\alpha_0}$ be an arbitrary constant map, for some
$\alpha_0<\omega_2$. Suppose that, for some $\delta<\omega_1$,
maps $f_\gamma$ are already defined for every $\gamma<\delta$ so
that $f_\gamma=i_\gamma r_\gamma$ for some embedding
$i_\gamma\colon r_\gamma(D)\to \tilde X_{\alpha_\gamma}$. If
$\delta$ is a limit ordinal, let
$\alpha_\delta=\sup\{\alpha_\gamma\mid\gamma<\delta\}$ and
$f_\delta=\varprojlim \{f_\gamma\mid\gamma<\delta\}$. If
$\delta=\delta'+1$, let $\alpha_\delta=\alpha_{\delta'}+1$ and
find an embedding $i_\delta\colon r_\delta(D)\to
(I^{\omega_1})^{\alpha_\delta}\subset \tilde X_{\alpha_\delta}$
such that $\tilde
p_{\alpha_{\delta}\alpha_{\delta'}}i_\delta=i_{\delta'}$ and
$\tilde
p_{\alpha_{\delta}\alpha_{\delta'}}i_\delta(d_\delta)=i_{\delta'}
(d_{\delta'})$. Put $f_\delta=i_\delta r_\delta$.

Finally, let $\alpha=\sup\{\alpha_\gamma\mid\gamma<\omega_1\}$ and
$f=\varprojlim
\{f_\gamma\mid\gamma<\omega_1\}$.
Claim is thus proved.

\smallskip

We now return to the proof of the theorem. Without loss of generality, we
assume that $D\subset (I^{\omega_1})^\alpha\subset \tilde
X_\alpha$ and $D$ is affinely independent in $\tilde X_\alpha$.
Recall that $h(D)$ denotes the closed convex hull of $D$ in
$\tilde X_\alpha$. We are going to show that the space
$(\cc(\tilde p_{\alpha+1,\alpha}))^{-1}(h(D))$ does not satisfy
the Souslin condition. There exist two maps
 $s_1,s_2\colon D\to \tilde X_{\alpha+1}$
such that $\tilde p_{\alpha+1,\alpha}s_1=\tilde
p_{\alpha+1,\alpha}s_2=1_{D}$ and $s_1(D)\cap s_2(D)=\emptyset$.
Let $U_1,U_2$ be neighborhoods of $s_1(D)$ and $s_2(D)$
respectively such that $\bar U_1\cap\bar U_2=\emptyset$.

For every isolated point $y\in D$ let $V_y$ be a neighborhood of
$y$ in $\tilde X_\alpha$ such that $\bar V_y\cap
h(D\setminus\{y\})=\emptyset$.

Let $$W_y=\langle \tilde X_{\alpha+1}\setminus (\bar U_2\cap\tilde
p_{\alpha+1,\alpha}^{-1}(D\setminus\{y\})), U_2\cap
\tilde p_{\alpha+1,\alpha}^{-1}(\bar V_y)\rangle.$$

We are going to show that $\cc(\tilde
p_{\alpha+1,\alpha})^{-1}(h(D))\cap W_y\neq\emptyset$. To this
end, consider the set $B=h(s_1(D\setminus\{y\})\cup\{s_2(y)\})$.
Obviously, $B\in\cc(\tilde p_{\alpha+1,\alpha})^{-1}(h(D))$ and
$s_2(y)\in B\cap U_2\cap \tilde p_{\alpha+1,\alpha}^{-1}(\bar
V_y)$. In addition, for every $z\in D\setminus\{d_{\omega_1}\}$,
$z\neq y$, we have $B\cap \tilde
p_{\alpha+1,\alpha}^{-1}(z)=\{s_1(z)\}$, therefore $B\subset
\tilde X_{\alpha+1}\setminus (\bar U_2\cap\tilde
p_{\alpha+1,\alpha}^{-1}(D\setminus\{y\}))$. We conclude that
$B\in W_y$.

It remains to prove that for every $y,z\in
D\setminus\{d_{\omega_1}\}$, $y\neq z$, we have $W_y\cap W_z\cap
\cc(\tilde p_{\alpha+1,\alpha})^{-1}(h(D))\neq\emptyset$. Indeed,
otherwise, for any $A\in W_y\cap W_z\cap \cc(\tilde
p_{\alpha+1,\alpha})^{-1}(h(D))$ we would have $A\cap \tilde
p_{\alpha+1,\alpha}^{-1}(y)\subset
p_{\alpha+1,\alpha}^{-1}(y)\setminus \bar U_2$ and, on the other
hand, $A\cap \tilde p_{\alpha+1,\alpha}^{-1}(y)\subset U_2$, a
contradiction. We therefore conclude that $$\{W_y\cap \cc(\tilde
p_{\alpha+1,\alpha})^{-1}(h(D))\mid y\in
D\setminus\{d_{\omega_1}\}\}$$ is a family of nonempty disjoint
open subsets in $\cc(\tilde p_{\alpha+1,\alpha})^{-1}(h(D))$.
Since the space $\cc(\tilde p_{\alpha+1,\alpha})^{-1}(h(D))$ does
not satisfy the Souslin condition, we obtain that $\cc(\tilde
p_{\alpha+1,\alpha})^{-1}(h(D))\not\in\ $AR and hence the map
$\cc(\tilde p_{\alpha+1,\alpha})$ is not a soft map. This
contradiction demonstrates that $\w(X)\le\omega_1$.

We are going to show that $X$ is openly generated. Since $\cc(X)$
is an AR of weight $\omega_1$, then there exists
 an
inverse system $\mathcal S=\{X_\alpha,p_{\alpha\beta};\omega_1\}$
consisting of compact metrizable convex spaces and affine maps
such that $\cc(X)=\varprojlim \cc(\mathcal S)$. Applying
Shchepin's Spectral Theorem, we may additionally assume that all
the maps $\cc(p_{\alpha\beta})$, $\beta\le\alpha<\omega_1$, are
soft. By Lemma \ref{l:op}, the maps $p_{\alpha\beta}$,
$\beta\le\alpha<\omega_1$, are soft and therefore open.

\end{proof}

\begin{thm}
Let $X$ be a convex compact subset of a locally convex space. Then the
space $\cc(X)$ is homeomorphic to $I^{\omega_1}$ if and only if
$X$ is homeomorphic to  $I^{\omega_1}$.
\end{thm}
\begin{proof}
We use the following characterization of the Tychonov cube
$I^\tau$, $\tau>\omega$, due to Shchepin \cite{S}: a compact
Hausdorff space $X$ of weight $\tau>\omega$ is homeomorphic to the
Tychonov cube $I^\tau$ if and only if $X$ is a character
homogeneous  absolute retract. Recall that  a space is called {\it
character homogeneous } if the characters of all of
its points are
equal.

If the weight of $X$ is $\omega_1$, then it easily follows from
the Shchepin Spectral Theorem \cite{S} that $X$ can be represented
as $\varprojlim\mathcal S$, where $\mathcal
S=\{X_\alpha,p_{\alpha\beta};\omega_1\}$ is an inverse system
consisting of convex compact metrizable subsets in locally convex
spaces and affine continuous maps. Since the functor $\cc$ is
continuous (see, e.g. \cite{N}), we obtain that
$\cc(X)=\varprojlim\{\cc(X_\alpha),\cc(p_{\alpha\beta});\omega_1\}$.
Since $\cc(X_\alpha)$ is an absolute retract (see \cite{IZ}) and,
by Proposition \ref{p:soft}, the map $\cc(p_{\alpha\beta})$ is
soft for every $\alpha,\beta<\omega_1$, $\alpha\ge\beta$, we apply
a result of Shchepin (see \cite{S}) to derive that $\cc(X)$ is an
absolute retract.

If $X$ is character homogeneous, then we can in addition assume
that no projection $p_{\alpha\beta}$  possesses one-point
preimages. By Lemma \ref{l:3}, the maps $\cc(p_{\alpha\beta})$ do
not possess one-point preimages and therefore $\cc(X)$ is character
homogeneous. By the mentioned result of Shchepin, $\cc(X)$ is
homeomorphic to $I^{\omega_1}$.

If $\cc(X)$ is homeomorphic to $I^{\omega_1}$, then there exists
 an
inverse system $\mathcal S=\{X_\alpha,p_{\alpha\beta};\omega_1\}$
consisting of compact metrizable convex spaces and open affine
maps such that $\cc(X)=\varprojlim \cc(\mathcal S)$. Applying
Shchepin's Spectral Theorem, we may additionally assume that all
the maps $\cc(p_{\alpha\beta})$, $\beta\le\alpha<\omega_1$, are
soft and do not possess points with one-point preimage. It is then
evident that the maps $p_{\alpha\beta}$,
$\beta\le\alpha<\omega_1$, do not possess points with one-point
preimage. Applying Lemma \ref{l:op} we conclude that the maps
$p_{\alpha\beta}$, $\beta\le\alpha<\omega_1$, are open and
therefore, by the Michael Selection Theorem, soft. Then $X$ is a
character homogeneous AR of weight $\omega_1$. By the cited
characterization theorem for $I^{\omega_1}$, the space $X$ is
homeomorphic to $I^{\omega_1}$.

\end{proof}

\section{Cone over Tychonov cube}

Define the {\it cone functor} $\cone$ in the category $\conv$ as
follows. Given an object $X$ in $\conv$, i.e. a compact convex
subset $X$ in a locally convex space $L$, let $\cone(X)$ be the
convex hull of the set $X\times\{0\}\cup\{(0,1)\}$ in
$L\times\BBR$. For a morphism $f\colon X\to Y$ in $\conv$ define
$\cone(f)\colon \cone(X)\to\cone(Y)$ as the only affine continuous
map that extends $f\times\{0\}\colon X\times\{0\}\to Y\times\{0\}$
and sends $(0,1)\in \cone(X)$ to $(0,1)\in \cone(Y)$.

We will need the following notion. A map $f\colon X\to Y$ is called
a {\em trivial $Q$-bundle} if there exists a homeomorphism $g\colon
X\to Y\times Q$ such that $f=\pr_1g$.
The following statement is a characterization theorem for the space
$\cone(I^{\omega_1})$ among the convex compact spaces.
\begin{prop}\label{p:cone}
A convex compactum $X$ is homeomorphic to the space
$\cone(I^{\omega_1})$ if and only if $X$ satisfies the properties:
\begin{enumerate}
  \item $X$ is an AR;
  \item $\w(X)=\omega_1$; and
  \item there exists a unique point $x\in X$ of countable
  character.
\end{enumerate}
\end{prop}
\begin{proof} Obviously, if a convex compactum $X$ is homeomorphic
to $\cone(I^{\omega_1})$, then $X$ satisfies properties 1)--3).

Suppose now that $X$ satisfies 1)--3). Then $X$ is homeomorphic to
the limit of a continuous inverse system $\mathcal
S=\{X_\alpha,p_{\alpha\beta};\omega_1\}$ in $\conv$ that satisfies
the properties
\begin{enumerate}
  \item[(i)] $X_\alpha$ is a convex metrizable compactum for every
  $\alpha$;
  \item[(ii)] $p_{\alpha\beta}$ is an open affine map for every
  $\alpha\ge\beta$; and
  \item[(iii)] $\{x_\beta\}=\{y\in X_\beta\mid\
|p_{\alpha\beta}^{-1}(y)|=1\}$.
\end{enumerate}
Passing, if necessary, to a subsystem of $\mathcal S$, one can
assume that for every $\alpha$ and every compact subset $K$ of
$X_\alpha\setminus\{x_\alpha\}$ the map
$$p_{\alpha+1,\alpha}|p_{\alpha+1,\alpha}^{-1}(K)\colon
p_{\alpha+1,\alpha}^{-1}(K)\to K$$ satisfies the condition of
fibrewise disjoint approximation. The Toru\'nczyk-West
characterization theorem \cite{TW} implies that, if $K$ is an AR,
the map $p_{\alpha+1,\alpha}|p_{\alpha+1,\alpha}^{-1}(K)$ is a
trivial $Q$-bundle and therefore the map
$$p_{\alpha+1,\alpha}|p_{\alpha+1,\alpha}^{-1}
(X_\alpha\setminus\{x_\alpha\})=p_{\alpha+1,\alpha}|
(X_{\alpha+1}\setminus\{x_{\alpha+1}\}),$$ being a locally trivial
$Q$-bundle, is a trivial $Q$-bundle (see \cite{C1}). Therefore,
the map $p_{\alpha+1,\alpha}$ is homeomorphic to the projection
map $\pr_{23}\colon Q\times Q\times [0,1)\to Q\times [0,1)$ (that
$X_\alpha\setminus\{x_\alpha\}$ is homeomorphic to $Q\times [0,1)$
follows from the fact that the spaces $Q$ and $\cone(Q)$ are
homeomorphic --  see \cite{C}). Passing to the one-point
compactifications of these maps we obtain the commutative diagram
$$\xymatrix{X_{\alpha+1}\ar[ddd]_{p_{\alpha+1,\alpha}}\ar[rrr]&
& &\cone(Q\times Q)\ar[ddd]^{\cone(\pr_2)}\\
&X_{\alpha+1}\setminus\{x_{\alpha+1}\}\ar[r]\ar@{^{(}->}[lu]
\ar[d]_{p_{\alpha+1,\alpha}|\dots}&Q\times
Q\times[0,1)\ar[d]^{\pr_{23}}\ar@{^{(}->}[ur]&
\\ &X_{\alpha}\setminus\{x_{\alpha\}}\ar@{^{(}->}[dl]\ar[r]&Q\times[0,1)\ar@{^{(}->}[dr]&\\
 X_{\alpha}\ar[rrr]& & & \cone(Q)}$$
in which the horizontal arrows are homeomorphisms. Therefore $X$
and $\cone(I^{\omega_1})$ are homeomorphic.
\end{proof}

\begin{thm}\label{t:cone}
Let $X$ be an object of the category $\conv$. Then the space $\cc(X)$
is homeomorphic to the cone over the Tychonov cube,
$\cone(I^{\omega_1})$, if and only if $X$ is homeomorphic to the
space $\cone(I^{\omega_1})$.
\end{thm}

\begin{proof} Suppose that a convex compact space $X$ is an
absolute retract of weight $\omega_1$ with exactly one point $x$,
of countable character. It follows from the Shchepin Spectral
Theorem (\cite{S}; see also \cite{Ch}) that $X$ can be represented
as $\varprojlim\mathcal S$, where $\mathcal
S=\{X_\alpha,p_{\alpha\beta};\omega_1\}$ is an inverse system in
which every $X_\alpha$ is a metrizable convex compactum and every
$p_{\alpha\beta}$, $\alpha\ge\beta$, is an affine map. Denote by
$p_\alpha\colon X\to X_\alpha$ the limit projections and let
$x_\alpha=p_\alpha(x)$. Passing, if necessary, to a subsystem of
$\mathcal S$, one can assume additionally that for every
$\alpha\ge\beta$ we have $\{x_\beta\}=\{y\in X_\beta\mid\
|p_{\alpha\beta}^{-1}(y)|=1\}$.

Then for every $\alpha\ge\beta$, the map $\cc(p_{\alpha\beta})$ is
a soft map and similarly as in the proof of Lemma \ref{l:3}, one
can show that $$\{\{x_\beta\}\}=\{A\in \cc(X_\beta)\mid\
|\cc(p_{\alpha\beta})^{-1}(A)|=1\}.$$ We conclude that the space
$\cc(X)=\varprojlim(\mathcal{S})$ satisfies the conditions of
Proposition \ref{p:cone} and therefore is homeomorphic to the
space $\cone(I^{\omega_1})$.

Now, if $\cc(X)$ is homeomorphic to $\cone(I^{\omega_1})$, it
follows from Theorem \ref{t:1} that $X$ is an AR of weight
$\omega_1$. Note that for every point $x$ of countable character
in $X$, the point $\{x\}$ is of countable character in $\cc(X)$.
We therefore conclude that there is a unique point of countable
character in $X$. By Proposition \ref{p:cone}, $X$ is homeomorphic
to $\cone(I^{\omega_1})$.
\end{proof}
\section{Remarks and open problems}
\begin{problem} Let $f\colon X\to Y$ be an affine continuous map
of compact metrizable compacta in locally convex spaces such that
$\dim f^{-1}(y)\ge2$, for every $y\in Y$. Is the map $\cc(f)\colon
\cc(X)\to\cc(Y)$ homeomorphic to the projection map
$\pr_1\colon Q\times Q\to
Q$?
\end{problem}

Note that there is an open map $f\colon X\to Y$ of metrizable
compacta with infinite fibers such that the map $P(f)\colon P(X)\to
P(Y)$ is not homeomorphic to the projection map $\pr_1\colon
Q\times Q\to Q$ (see \cite{D}). (Recall that $P$ denotes the
probability measure functor).
\begin{problem}
Does every compact convex AR of weight $\tau\ge\omega_1$ contain an
affine copy of the Tychonov cube $I^\tau$?
\end{problem}

It is known that every compact Hausdorff AR of weight
$\tau\ge\omega_1$ contains a topological copy of the Tychonov cube
$I^\tau$ (see \cite{S1}).

The theory of nonmetrizable noncompact absolute extensors which
is, in some sense, parallel to that of compact absolute extensors,
was elaborated by Chigogidze \cite{Ch}\cite{Ch1}. One can also consider
the hyperspaces of compact subsets in the spaces $\mathbb R^\tau$
and conjecture that, for noncountable $\tau$, the hyperspace
$\cc(\mathbb R^\tau)$ is homeomorphic to $\mathbb R^\tau$ if and
only if $\tau=\omega_1$.

\end{document}